\documentclass[12pt]{amsart}
\usepackage{latexsym,amssymb,amscd}
\usepackage{graphicx}
\def\B'c{{\mathcal{B'}}}
\def\U'c{{\mathcal{U'}}}

\def\opn#1#2{\def#1{\operatorname{#2}}} 
\opn\chara{char}
\opn\length{\ell}
\opn\cd{cd}
\opn\projdim{pd}
\opn\injdim{inj\,dim}
\opn\ini{in}
\opn\rank{rank}
\opn\depth{depth}
\opn\height{ht}
\opn\bigheight{bight}
\opn\embdim{emb\,dim}
\opn\codim{codim}

\opn\Tr{Tr}
\opn\bigrank{big\,rank}
\opn\superheight{superheight}\opn\lcm{lcm}
\opn\trdeg{tr\,deg}%
\opn\reg{reg}
\opn\lreg{lreg}
\opn\set{set}
\opn\supp{Supp}
\opn\bight{bight}
\opn\shad{Shad}
\opn\indeg{indeg}
\opn\lex{lex}
\opn\div{div}
\opn\Div{Div}
\opn\cl{cl}
\opn\Cl{Cl}
\opn\Spec{Spec}
\opn\Supp{Supp}
\opn\supp{supp}
\opn\Sing{Sing}
\opn\Ass{Ass}
\opn\Ann{Ann}
\opn\Rad{Rad}
\opn\Soc{Soc}
\opn\Ker{Ker}
\opn\Coker{Coker}
\opn\Im{Im}
\opn\Hom{Hom}
\opn\Tor{Tor}
\opn\Ext{Ext}
\opn\End{End}
\opn\Aut{Aut}
\opn\id{id}

\opn\nat{nat}
\opn\GL{GL}
\opn\SL{SL}
\opn\mod{mod}
\opn\ord{ord}
\opn\ara{ara}
\opn\aff{aff}
\opn\con{conv}
\opn\relint{relint}
\opn\st{st}
\opn\lk{lk}
\opn\cn{cn}
\opn\core{core}
\opn\vol{vol}
\opn\gr{gr}

\def\pot#1#2{#1[\kern-0.28ex[#2]\kern-0.28ex]}

\opn\dirlim{\underrightarrow{\lim}}
\opn\invlim{\underleftarrow{\lim}}
\def\pnt{{\raise0.5mm\hbox{\large\bf.}}}

\def\Implies{\ifmmode\Longrightarrow \else
	\unskip${}\Longrightarrow{}$\ignorespaces\fi}
\def\implies{\ifmmode\Rightarrow \else
	\unskip${}\Rightarrow{}$\ignorespaces\fi}
\def\iff{\ifmmode\Longleftrightarrow \else
	\unskip${}\Longleftrightarrow{}$\ignorespaces\fi}

\let\:=\colon
\newtheorem{Theorem}{Theorem}[section]
\newtheorem{Lemma}[Theorem]{Lemma}
\newtheorem{Corollary}[Theorem]{Corollary}
\newtheorem{Proposition}[Theorem]{Proposition}
\newtheorem{Remark}[Theorem]{Remark}

\let\epsilon=\varepsilon
\let\phi=\varphi
\let\kappa=\varkappa
\numberwithin{equation}{section}

\title{On the closed neighborhood ideal of the square of the path graph}
 \author{ Anda Olteanu \and Oana Olteanu}

\address{Faculty of Marine Engineering, ``Mircea cel B\u atr\^an" Naval Academy, Fulgerului Street, no. 1,
	900218 Constanta, Romania,} \email{olteanuandageorgiana@gmail.com}
\address{Faculty of Applied Sciences,
	National University of Science and Technology Politehnica Bucharest, 
	Splaiul Independen\c tei, No.
	313, 060042, Bucharest, Romania}\email{olteanuoanastefania@gmail.com} 
\begin{document}
	
	\begin{abstract} We consider the closed neighborhood ideal of square of the path graph and study its invariants. We compute the height, the projective dimension and the Castelnuovo--Mumford regularity. We prove that these ideals are sequentially Cohen--Macaulay and characterize when they are Cohen--Macaulay.
		
		\textbf{Keywords}: Castelnuovo--Mumford regularity, projective dimension, simplicial complex\\ 
		
		\textbf{MSC}: Primary 13D05 Secondary 13F55
	\end{abstract}

	
	\maketitle
	
	\section*{Introduction}
Squarefree monomial ideals are a central theme of research in both commutative algebra and algebraic combinatorics. It is common to express invariants of squarefree monomial ideals in terms of some combinatorial structures that can be attached to them (simplicial complexes, hypergraphs or graphs). 

In 2020, Sharifan and Moradi defined the notion of closed neighborhood ideal of a graph as being the squarefree monomial ideal whose generators are given by the closed neighborhoods of the vertices \cite{SM}. They studied several invariants such as the projective dimension, the height, and the Castelnuovo--Mumford regularity and relate them to combinatorial invariants such as the domination number or the matching number. Recently, they continued their study with chordal graphs \cite{MS1}. Closed neighborhood ideals of graphs have been studied also by Joseph, Roy, and Singh studied the minimal free resolution in the framework of the Barile--Macchia resolutions \cite{JRS}. They proved that for trees, the Barile--Macchia resolution is minimal. These ideals have been studied also for particular classes of graphs. For instance, Nasernejad, Bandari, and Roberts proved that the closed neighborhood ideals of complete bipartite graphs are normal, therefore they satisfy the (strong) persistence property \cite{NBR}. Moreover, Chakraborty, Joseph, Roy, and Singh \cite{CJRS} considered the Castelnuovo--Mumford regularity of closed neighborhood ideals of a forest $G$ and obtained that it is equal to the matching number of $G$, fact that was conjectured by Sharifan and Moradi \cite[Conjecture 2.11]{SM}.

 In this paper we pay attention to the closed neighborhood ideal of the square of path graph, $NI(P_n^2)$. We determine the height, the bigheight, the projective dimension, and the Castelnuovo--Mumford regularity, and we characterize all closed neighborhood ideals of $P_n^2$ which are Cohen--Macaulay. Moreover, if we assume that $NI(P_n^2)\subseteq S=k[x_1,\ldots,x_n]$, we prove that $S/NI(P_n^2)$ is sequentially Cohen--Macaulay for $n\geq7$. A similar result was obtained by Sharifan and Moradi for complete $r$-partite graphs \cite[Theorem 2.10 ((vii))]{SM}. We also pay attention to the simplicial complex  $\mathcal{NI}[P_n^2]$ whose facet ideal is $NI(P_n^2)$ and prove that it is shellable. Due to its nice structure, we are able to compute its $f$-vector and the reduced Euler characteristic.
 
 The paper is structured as follows: the first section is devoted to the study of algebraic invariants for the closed neighborhood ideal such as the height (Proposition \ref{height}), the projective dimension, and the Castelnuovo--Mumford regularity (Theorem \ref{pdreg}). As a consequence, we prove that, for closed neighborhood ideals of $P_n^2$, the depth and the Castelnuovo--Mumford regularity are equal Corollary \ref{depth-reg}. We also give a complete characterization of all the closed neighborhood ideals of $P_n^2$ which are Cohen--Macaulay (Corollary \ref{CM}).
 In the second section we study closed neighborhood ideals from the combinatorial point of view, by relating them to simplicial complexes. We consider the simplicial complex $\mathcal{NI}[P_n^2]$ such that the closed neighborhood ideal of $P_n^2$ is the facet ideal of $\mathcal{NI}[P_n^2]$. We prove that the simplicial complex is shellable (Proposition \ref{shellable}) and it has the free vertex property (Proposition \ref{freevertex}). This allows us to conclude that $S/NI(P_n^2)$ is sequentially Cohen--Macaulay (Theorem \ref{SCM1}). Finally, we compute the $f$-vector of $\mathcal{NI}[P_n^2]$ (Proposition \ref{fvector}) and, as a consequence, we obtain the reduced Euler characteristic.

	\section{Invariants of the closed neighborhood ideals of $P_n^2$}

	In this section we pay attention to closed neighborhood ideals of $P_n^2$. For this class of ideals we will compute the height, the projective dimension and the Castelnuovo--Mumford regularity. We start by recalling the notions that will be used (for more details, see \cite{BH} or \cite{V}).
	
	The path graph $P_n$ is the graph on the vertex set $V(P_n)=\{x_1,\ldots,x_n\}$ and with the set of edges $E(P_n)=\{\{x_i,x_{i+1}\}:1\leq i\leq n-1\}$. The square of the path graph is denoted by $P_n^2$ and is the graph with the same vertex set as $P_n$ and the set of edges $E(P_n^2)=E(P_n)\cup\{\{x_i,x_{i+2}\}:1\leq i\leq n-2\}$. In the figure bellow there is the graph $P_7$ and its square:
	 
	 \begin{center}
	 	\begin{figure}[h]
	 		\includegraphics[height=1cm]{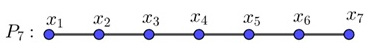}$\qquad$ \includegraphics[height=2cm]{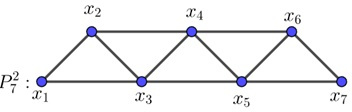}	\caption{The graphs $P_7$ and $P_7^2$}\label{P72graf}
	 	\end{figure}
	 \end{center}
	 
	 For a given finite simple graph $G$,\textit{ the neighborhood} of a vertex $x$ is the set $$N(x)=\{y\in V(G):\{x,y\}\in E(G)\}.$$ we denote by $N[x]=N(x)\cup\{x\}$ \textit{the closed neighborhood} of $x$. A set $M\subseteq V(G)$ is called \textit{a dominating set} is $M\cap N[x]\neq\emptyset$ for any vertex $x\in V(G)$ and it is called a \textit{minimal dominating set} if none of its proper subsets is a dominating set. \textit{The domination number} of $G$ is denoted by $\gamma(G)$ and is the minimum size of a dominating set of $G$. We denote by $\gamma'_G$ the maximal size of a minimal dominating set of $G$.
	 
	 For a finite, simple graph $G$ with $V(G)=\{x_1,\ldots,x_n\}$ let $S=k[x_1,\ldots,x_n]$ be the polynomial ring over a field $k$. (The same notation will be used for both vertices and variables where there will be no confusion from the context) For a set of vertices $F\subseteq V(G)$ we consider the squarefree monomial $\mathbf{x}_F=\prod\limits_{x_i\in F}x_i$. We will refer to $F$ as \textit{the support of the monomial} $\mathbf{x}_F$ and we will denote it by $\supp(\mathbf{x}_F)$. Sharifan and Moradi  defined the notion of closed neighborhood ideal, denoted by $NI(G)$, as being the squarefree monomial ideal generated by the monomials which correspond to the closed neighborhoods of vertices (\cite[Definition 2.1]{SM}): $$NI(G)=\left(\mathbf{x}_{N[x_i]}:x_i\in V(G) \right).$$
	 For instance, taking into account Figure \ref{P72graf}, we get
	 
	 $$NI(P_7^2)=\left(\mathbf{x}_{N[x_1]}=x_1x_2x_3,\mathbf{x}_{N[x_2]}=x_1x_2x_3x_4,\mathbf{x}_{N[x_3]}= x_1x_2x_3x_4x_5,\right.$$ $$\left.\mathbf{x}_{N[x_4]}=x_2x_3x_4x_5x_6,\mathbf{x}_{N[x_5]}=x_3x_4x_5x_6x_7,\mathbf{x}_{N[x_6]}=x_4x_5x_6x_7,\mathbf{x}_{N[x_7]}=x_5x_6x_7\right)$$ that is
	 
	 $$NI(P_7^2)=(x_1x_2x_3,x_2x_3x_4x_5x_6,x_5x_6x_7).$$
	Since the path graphs that we use are defined on various set of vertices, throughout this section, we will denote by $P_n$ the path on $\{x_1,\ldots,x_n\}$ and by $P_{\{i,\ldots,n\}}$ the path on the set of vertices $\{x_i,\ldots,x_n\}$, for $1<i<n$. Also, $S=k[x_1,\ldots,x_n]$ will be the polynomial ring in $n$ variables over a field $k$. Otherwise, the polynomial ring will be described explicitly.
	
	  Since the shape of the closed neighborhood ideal depends on $n$, we will firstly consider $n\leq 6$.
	
	\begin{Proposition}\label{n<=6}
		Let $3\leq n\leq 6$ be an integer and $NI(P_n^2)$ be the closed neighborhood ideal. The following hold:
		\begin{itemize}
			\item[(i)] $\height(NI(P_3^2)=1$, $\reg(S/NI(P_3^2))=2$, and $\projdim(S/NI(P_3))^2=1$ .
			\item[(ii)] $\height(NI(P_4^2)=1$, $\reg(S/NI(P_4^2))=2$, and $\projdim(S/NI(P_4))^2=2$.
			\item[(iii)] $\height(NI(P_5^2)=1$, $\reg(S/NI(P_5^2))=3$, and $\projdim(S/NI(P_5))^2=2$.
			\item[(iv)] $\height(NI(P_6^2)=2$, $\reg(S/NI(P_6^2))=4$, and $\projdim(S/NI(P_6))^2=2$ .  
		\end{itemize}
	\end{Proposition}
	
	\begin{proof}
		(i) Since $NI(P_3^2)=(x_1x_2x_3)$, the statement follows.
		
		(ii) In this case $NI(P_4^2)=(x_1x_2x_3,x_2x_3x_4)=x_2x_3(x_1,x_4)$. The statement is straightforward.
		
		(iii) We have $NI(P_5^2)=(x_1x_2x_3,x_3x_4x_5)=x_3(x_1x_2,x_4x_5)$, therefore the formulae hold.
		
			(iv) We have $NI(P_5^2)=(x_1x_2x_3,x_4x_5x_6)$, so it is a complete intersection. The statement follows.
	\end{proof}
	Let $n\geq7$ be an integer, $NI(P^2_{n})\subset S=k[x_1,\ldots,x_n]$ be the closed neighborhood ideal of the square of the path graph $P_n$ with the set of minimal monomial generators
	$$\{m_1=x_1x_2x_3,m_{n-4}=x_{n-2}x_{n-1}x_n\}\cup\{m_i=x_ix_{i+1}x_{i+2}x_{i+3}x_{i+4}: 2\leq i\leq n-5\}.\ \ \ (*)$$
	
	Firstly we will compute the height of $NI(P_n^2)$.
	\begin{Proposition}\label{height}{\rm{(The height of $NI(P_n^2)$)}} Let $n\geq3$ be an integer. Then  
		$$\height(NI(P^2_{n}))=\left\lceil\frac{n}{5}\right\rceil.$$
	\end{Proposition}

\begin{proof} Note that, for $n\leq6$, the formula hold by Proposition \ref{n<=6}. We assume now that $n\geq7$, so the generators are of the form given in $(*)$.
	
	We will analyze the following cases:
	
	\textit{Case 1}: Assume that $n\equiv 0 (\mod\ 5)$, that is $n=5s$, for some integer $s\geq 1$. We claim that a minimal prime ideal of $NI(P^2_{n})$, where
	$$\mathcal{G}(NI(P^2_{n}))=\{m_1=x_1x_2x_3,m_{5s-4}=x_{5s-2}x_{5s-1}x_{5s}\}\cup$$
	$$\cup\{m_i=x_ix_{i+1}x_{i+2}x_{i+3}x_{i+4}: 2\leq i\leq 5s-5\}$$ is 
	$$\frak p=(x_3,x_8,\ldots,x_{5s-2})=(x_{5k+3}:0\leq k\leq s-1).$$
	Indeed, it is clear that $x_3\mid m_1$ and $x_{5s-2}\mid m_{5s-4}$. For every $2\leq i\leq 5s-5$, since the monomial $m_i$ has the support consisting of five consecutive variables $x_i,\ldots, x_{i+4}$, we have that there exists $0\leq k\leq s-1$ such that $x_{5k+3}\mid m_i$. Therefore we obtain that $\frak p\supseteq NI(P^2_{n})$ and $\height NI(P^2_{n})\leq \height\frak p=s=[\frac{n}{5}]$.
	
	Assume that there exists another minimal prime ideal $\frak q$ of $NI(P^2_{n})$ such that $\height \frak q<s.$ Since $m_1=x_1x_2x_3$ and $m_{5s-4}=x_{5s-2}x_{5s-1}x_{5s}$ are minimal monomial generators of $NI(P^2_{n})$ there exist $1\leq j\leq 3$ and $5s-2\leq k\leq 5s$ such that $\frak q\supseteq(x_j,x_k)$. Moreover, 
	$$\{m_{4}=x_4x_5x_6x_7x_8,\ldots, m_{5s-7}=x_{5s-7}x_{5s-6}x_{5s-5}x_{5s-4}x_{5s-3}\}\subseteq \mathcal{G}(NI(P_{n})),$$
	therefore $\frak q$ contains at least one variable from each of disjoint sets:
	$$A_{1}=\{x_4,x_5,x_6,x_7,x_8\}, A_{2}=\{x_9,x_{10},x_{11},x_{12},x_{13}\},\ldots$$
	$$ A_{s-2}=\{x_{5s-10},x_{5s-9},x_{5s-8},x_{5s-7},x_{5s-6}\}.$$
	Hence $\height \frak q\geq 2+(s-2)=s$, contradiction.
	
	\textit{Case 2}: Assume that $n\equiv 1,2,3 (\mod\ 5)$, that is $n=5s+r$, $1\leq r\leq 3$, for some integer $s\geq 1$. We claim that a minimal prime ideal of $NI(P^2_{n})$ is 
	$$\frak p=(x_3,x_8,\ldots,x_{5s-2},x_{5s+1})=(x_{5k+3}:0\leq k\leq s-1)+(x_{5s+1}).$$
	One has that $x_3\mid m_1$ and $x_{5s+1}\mid m_{n-4}$. For every $2\leq i\leq n-5$ there exists $0\leq k\leq s-1$ such that $x_{5k+3}\mid m_i$ since the monomial $m_i$ has the support consisting of five consecutive variables. Therefore, also in this case, we obtain that $\frak p\supseteq NI(P^2_{n})$ and $\height NI(P^2_{n})\leq \height\frak p=s+1=\lceil\frac{n}{5}\rceil$.
	
	Assume that there exists another minimal prime ideal $\frak q$ of $NI(P^2_{n})$ such that $\height \frak q\leq s.$ Similar to previous case, $m_1=x_1x_2x_3$ and $m_{n-4}=x_{n-2}x_{n-1}x_{n}$ are minimal monomial generators of $NI(P^2_{n})$, therefore there exist $1\leq j\leq 3$ and $n-2\leq l\leq n$ such that $\frak q\supseteq(x_j,x_l)$. 
	
	We have that 
	$$\{m_{4}=x_4x_5x_6x_7x_8,\ldots, m_{n-7}=x_{n-7}x_{n-6}x_{n-5}x_{n-4}x_{n-3}\}\subseteq \mathcal{G}(I),$$
	therefore $\frak q$ contains at least one variable from each of disjoint sets:
	$$A_{1}=\{x_4,x_5,x_6,x_7,x_8\}, A_{2}=\{x_9,x_{10},x_{11},x_{12},x_{13}\},\ldots$$
	$$ A_{s-1}=\{x_{n-7},x_{n-6},x_{n-5},x_{n-4},x_{n-3}\}.$$
	Hence $\height \frak q\geq 2+(s-1)=s+1$, contradiction.
	
		\textit{Case 3}: Assume that $n\equiv 4 (\mod\ 5)$, that is $n=5s+4$, for some integer $s\geq 1$. We claim that a minimal prime ideal of $NI(P^2_{n})$ is 
	$$\frak p=(x_3,x_8,\ldots,x_{5s-2},x_{5s+2})=(x_{5k+3}:0\leq k\leq s-1)+(x_{5s+2}).$$
	The proof is analogue to \textit{Case 2.}
\end{proof}


	



In \cite{SM}, a connection between the height of the closed neighborhood ideal of a graph $G$ and the dominating number of $G$ is established:

\begin{Proposition}\label{domnum}\cite[Lemma 2.2]{SM} Let $G$ be a graph. Then $\height(NI(G))=\gamma(G)$ and $\bigheight(G)=\gamma'_G$.
	
\end{Proposition}

Therefore we also obtain the dominating number of $P_n^2$:
\begin{Corollary}\label{domnumpn2}
	Let $n\geq 3$ be an integer. Then $\gamma(P_n^2)=\left\lceil\frac{n}{5}\right\rceil.$
\end{Corollary}
\begin{proof}
	The proof is a direct consequence of Propositions \ref{height} and \ref{domnum}. 
\end{proof}
Next we pay attention to the Castelnuovo--Mumford regularity and to the projective dimension. We will assume that $n\geq 7$, therefore the closed neighborhood ideal of $P_n^2$ can be written as
$$NI(P^2_{n})=(x_1x_2x_3)+I_5(P_{\{2,\ldots,n-1\}})+(x_{n-2}x_{n-1}x_n)$$
where by $I_t(P_n)$ we mean the ideal of the paths of length $t$ of the path graph $P_n$. 
\begin{Remark}\rm We will use the expression $n=cp+d$ for emphasizing that $d$ is the remainder, $0\leq d\leq p-1$. 
	\end{Remark}
	We recall that the projective dimension and the Castelnuovo-Mumford regularity of path ideals of paths were given in \cite[Corollary 4.15]{HeT}:

\begin{Lemma}\cite[Corollary 4.15]{HeT}\label{pdpath}
	Let $n,t,p,d$ be integers such that $n\geq 2$, $2\leq t\leq n$, $n=p(t+1)+d$, where $p\geq 0$ and $0\leq d\leq t$. Then
	\begin{itemize}
		\item the projective dimension of the path ideal of a path $P_n$ is given by 
		$$\projdim(S/I_t(P_n))=\begin{cases}
			2p,&d\neq t\\
			2p+1,&d=t
		\end{cases}$$
		\item the regularity of the path ideal of a path $P_n$ is given by 
		$$\reg(S/I_t(P_n))=\begin{cases}
			p(t-1),&d<t\\
			(p+1)(t-1),&d=t
		\end{cases}$$
	\end{itemize}
\end{Lemma}

The following result of Sharifan will be a key point in our study:

\begin{Lemma}\cite[Corollary 2.6]{S} \rm If $I$ is a monomial ideal of $S$ and $u$ is a monomial which does not belong to $I$, then the minimal free resolution of $S/(I,u)$ is given by the mapping cone technique provided that there is $x_i\in\supp(u)$ such that for all $v\in \mathcal{G}(I)$ $\deg_{x_i} (u) > \deg_{x_i}(v)$.
\end{Lemma}
So, in order to compute the two invariants, we will apply twice the mapping cone technique. Firstly we will apply it to the short exact sequence
$$0\rightarrow \frac{S}{I_5(P_{\{2,\ldots,n-1\}}):(x_1x_2x_3)}(-3)\rightarrow\frac{S}{I_5(P_{\{2,\ldots,n-1\}})}\rightarrow \frac{S}{I_5(P_{\{2,\ldots,n-1\}})+(x_1x_2x_3)}\rightarrow 0.$$

Then, we will apply the mapping cone technique to the short exact sequence
$$0\rightarrow \frac{S}{(I_5(P_{\{2,\ldots,n-1\}}),x_1x_2x_3):(x_{n-2}x_{n-1}x_n)}(-3)\rightarrow\frac{S}{(I_5(P_{\{2,\ldots,n-1\}}),x_1x_2x_3)}\rightarrow $$ $$\rightarrow\frac{S}{NI(P^2_{n})}\rightarrow 0.$$
We start with the first exact sequence and we determine the Castelnuovo--Mumford regularity and the projective dimension for  $\frac{S}{I_5(P_{\{2,\ldots,n-1\}})+(x_1x_2x_3)}$. By \cite[Theorem 2.4]{S} applied in the first exact sequence, we have:
\begin{Lemma}\label{reg-1} In the above notations,
$$\reg\left(\frac{S}{I_5(P_{\{2,\ldots,n-1\}})+(x_1x_2x_3)}\right)=$$
$$=\max\left\{\reg\left(\frac{S}{I_5(P_{\{2,\ldots,n-1\}}):(x_1x_2x_3)}\right)+2,\reg\left(\frac{S}{I_5(P_{\{2,\ldots,n-1\}})}\right)\right\}$$
and $$\projdim\left(\frac{S}{I_5(P_{\{2,\ldots,n-1\}})+(x_1x_2x_3)}\right)=$$
$$=\max\left\{\projdim\left(\frac{S}{I_5(P_{\{2,\ldots,n-1\}}):(x_1x_2x_3)}\right)+1,\projdim\left(\frac{S}{I_5(P_{\{2,\ldots,n-1\}})}\right)\right\}.$$
\end{Lemma}

\begin{Remark}\rm
One may note that if we replace $t=5$, and $n-2=6p'+d'$ in Lemma \ref{pdpath}, we get
$$\reg\left(\frac{S}{I_5(P_{\{2,\ldots,n-1\}})}\right)=\reg\left(\frac{k[x_2,\ldots,x_{n-1}]}{I_5(P_{\{2,\ldots,n-1\}})}\right)=\begin{cases}
	4p',&d'<5\\
	4(p'+1),&d'=5
\end{cases}$$
and $$\projdim\left(\frac{S}{I_5(P_{\{2,\ldots,n-1\}})}\right)=\projdim\left(\frac{k[x_2,\ldots,x_{n-1}]}{I_5(P_{\{2,\ldots,n-1\}})}\right)=\begin{cases}
	2p',&d'<5\\
	2p'+1,&d'=5
\end{cases}.$$

\end{Remark}
\begin{Lemma}\label{reg1}
	For $n=6p+d$, with $p\geq 1$, one has
	$$\reg\left(\frac{S}{I_5(P_{\{2,\ldots,n-1\}})+(x_1x_2x_3)}\right)=\begin{cases}
		4p-2,&d=0\\
		4p,&d\in\{1,2,3\}\\
		4p+2,&d\in\{4,5\}
	\end{cases}$$
	and $$\projdim\left(\frac{S}{I_5(P_{\{2,\ldots,n-1\}})+(x_1x_2x_3)}\right)=\begin{cases}
		2p-1&d=0\\
		2p,&d\in\{1,2,3\}\\
		2p+1,&d\in\{4,5\}
	\end{cases}$$
	\end{Lemma}
	
	\begin{proof}
		We will prove by induction on $n$. As mentioned before,
		$$\reg\left(\frac{S}{I_5(P_{\{2,\ldots,n-1\}})+(x_1x_2x_3)}\right)=$$
		$$=\max\left\{\reg\left(\frac{S}{I_5(P_{\{2,\ldots,n-1\}}):(x_1x_2x_3)}\right)+2,\reg\left(\frac{S}{I_5(P_{\{2,\ldots,n-1\}})}\right)\right\}$$
		$$\projdim\left(\frac{S}{I_5(P_{\{2,\ldots,n-1\}})+(x_1x_2x_3)}\right)=$$
		$$=\max\left\{\projdim\left(\frac{S}{I_5(P_{\{2,\ldots,n-1\}}):(x_1x_2x_3)}\right)+1,\projdim\left(\frac{S}{I_5(P_{\{2,\ldots,n-1\}})}\right)\right\}$$
		and $$\reg\left(\frac{S}{I_5(P_{\{2,\ldots,n-1\}})}\right)=\reg\left(\frac{k[x_2,\ldots,x_{n-1}]}{I_5(P_{\{2,\ldots,n-1\}})}\right)=\begin{cases}
			4p',&d'<5\\
			4(p'+1),&d'=5
		\end{cases}$$
		$$\projdim\left(\frac{S}{I_5(P_{\{2,\ldots,n-1\}})}\right)=\projdim\left(\frac{k[x_2,\ldots,x_{n-1}]}{I_5(P_{\{2,\ldots,n-1\}})}\right)=\begin{cases}
			2p',&d'<5\\
			2p'+1,&d'=5
		\end{cases}$$where $n-2=6p'+d'$. Therefore, for $n=6p+d$ we have
		$$\reg\left(\frac{S}{I_5(P_{\{2,\ldots,n-1\}})}\right)=\begin{cases}
			4p,&d\in\{1,2,3,4,5\}\\
			4(p-1),&d=0
		\end{cases}$$
		and $$\projdim\left(\frac{S}{I_5(P_{\{2,\ldots,n-1\}})}\right)=\begin{cases}
			2p,&d\in\{2,3,4,5\}\\
			2p-2,&d=0\\
			2p-1,&d=1
		\end{cases}$$
	 
		One may note that 
		$$I_5(P_{\{2,\ldots,n-1\}}):(x_1x_2x_3)=\begin{cases}
			(x_4x_5x_6),&n\in\{7,8,9\}\\
			I_5(P_{\{5,\ldots,n-1\}})+(x_4x_5x_6),&n\geq 10
		\end{cases}$$
		In this case we have $$\reg\left(\frac{S}{I_5(P_{\{2,\ldots,n-1\}}):(x_1x_2x_3)}\right)=\begin{cases}
			1,&n\in\{7,8,9\}\\
			\reg\left(\frac{S}{I_5(P_{\{5,\ldots,n-1\}})+(x_4x_5x_6)}\right),&n\geq 10
			\end{cases}$$
			and $$\projdim\left(\frac{S}{I_5(P_{\{2,\ldots,n-1\}}):(x_1x_2x_3)}\right)=\begin{cases}
				1,&n\in\{7,8,9\}\\
				\projdim\left(\frac{S}{I_5(P_{\{5,\ldots,n-1\}})+(x_4x_5x_6)}\right),&n\geq 10
			\end{cases}$$
			By induction we obtain
			$$\reg\left(\frac{S}{I_5(P_{\{5,\ldots,n-1\}})+(x_4x_5x_6)}\right)=\reg\left(\frac{k[x_4,\ldots,x_n]}{I_5(P_{\{5,\ldots,n-1\}})+(x_4x_5x_6)}\right)=$$ $$=\begin{cases}
				4p''-2,&d''=0\\
				4p'',&d''\in\{1,2,3\}\\
				4p''+2,&d''\in\{4,5\}
			\end{cases}$$
			and$$\projdim\left(\frac{S}{I_5(P_{\{5,\ldots,n-1\}})+(x_4x_5x_6)}\right)=\projdim\left(\frac{k[x_4,\ldots,x_n]}{I_5(P_{\{5,\ldots,n-1\}})+(x_4x_5x_6)}\right)=$$
			$$=\begin{cases}
				2p''-1,&d''=0\\
				2p'',&d''\in\{1,2,3\}\\
				2p''+1,&d''\in\{4,5\}
			\end{cases}$$
			where $n-3=6p''+d''$.
			Using the fact that $n=6p+d$, we obtain
				$$\reg\left(\frac{S}{I_5(P_{\{5,\ldots,n-1\}})+(x_4x_5x_6)}\right)=\begin{cases}
				4p-4,&d=0\\
				4p-2,&d\in\{1,2,3\}\\
				4p,&d\in\{4,5\}
			\end{cases}$$
			and $$\projdim\left(\frac{S}{I_5(P_{\{5,\ldots,n-1\}})+(x_4x_5x_6)}\right)=\begin{cases}
				2p-2,&d=0\\
				2p-1,&d\in\{1,2,3\}\\
				2p,&d\in\{4,5\}
			\end{cases}$$
		Since 	$$\reg\left(\frac{S}{I_5(P_{\{2,\ldots,n-1\}})+(x_1x_2x_3)}\right)=$$
		$$=\max\left\{\reg\left(\frac{S}{I_5(P_{\{2,\ldots,n-1\}}):(x_1x_2x_3)}\right)+2,\reg\left(\frac{S}{I_5(P_{\{2,\ldots,n-1\}})}\right)\right\}$$
		we obtain that $$\reg\left(\frac{S}{I_5(P_{\{2,\ldots,n-1\}})+(x_1x_2x_3)}\right)=\begin{cases}
			4p-2,&d=0\\
			4p,&d\in\{1,2,3\}\\
			4p+2,&d\in\{4,5\}
		\end{cases}$$as desired.
		Moreover, by $$\projdim\left(\frac{S}{I_5(P_{\{2,\ldots,n-1\}})+(x_1x_2x_3)}\right)=$$
		$$=\max\left\{\projdim\left(\frac{S}{I_5(P_{\{2,\ldots,n-1\}}):(x_1x_2x_3)}\right)+1,\projdim\left(\frac{S}{I_5(P_{\{2,\ldots,n-1\}})}\right)\right\}$$
		it results that 
		$$\projdim\left(\frac{S}{I_5(P_{\{2,\ldots,n-1\}})+(x_1x_2x_3)}\right)=\begin{cases}
			2p-1&d=0\\
			2p,&d\in\{1,2,3\}\\
			2p+1,&d\in\{4,5\}
		\end{cases}$$
	\end{proof}
Now we will use the second exact sequence:
\begin{Theorem}\label{pdreg}{\rm{(The Castelnuovo--Mumford regularity and the projective dimension)}}
	
	If $n=6p+d$, then
	$$\reg\left(\frac{S}{NI(P^2_{n})}\right)=\begin{cases}
		4p,&d\in\{0,1\}\\
		4p+1,&d=2\\
		4p+2,&d\in\{3,4\}\\
		4p+3,&d=5
	\end{cases}$$
	and $$\projdim\left(\frac{S}{NI(P^2_{n})}\right)=\begin{cases}
		2p,&d=0\\
		2p+1,&d\in\{1,2,3\}\\
		2p+2,&d\in\{4,5\}
	\end{cases}$$
\end{Theorem}

\begin{proof}Firstly, we note that, by Proposition \ref{n<=6}, the statement is true for $3\leq n\leq 6$. Therefore, we assume that $n\geq7$.
	We will apply the mapping cone technique to the short exact sequence

	$$0\rightarrow \frac{S}{(I_5(P_{\{2,\ldots,n-1\}}),x_1x_2x_3):(x_{n-2}x_{n-1}x_n)}(-3)\rightarrow\frac{S}{(I_5(P_{\{2,\ldots,n-1\}}),x_1x_2x_3)}\rightarrow $$ $$\rightarrow\frac{S}{NI(P^2_{n})}\rightarrow 0.$$
	As before, one has  $$\reg\left(\frac{S}{NI(P^2_{n})}\right)=$$
	$$=\max\left\{\reg\left(\frac{S}{(I_5(P_{\{2,\ldots,n-1\}}),x_1x_2x_3):(x_{n-2}x_{n-1}x_n)}\right)+2,\reg\left(\frac{S}{(I_5(P_{\{2,\ldots,n-1\}}),x_1x_2x_3)}\right)\right\}$$
	and
	 $$\projdim\left(\frac{S}{NI(P^2_{n})}\right)=$$
	$$=\max\left\{\projdim\left(\frac{S}{(I_5(P_{\{2,\ldots,n-1\}}),x_1x_2x_3):(x_{n-2}x_{n-1}x_n)}\right)+1,\projdim\left(\frac{S}{(I_5(P_{\{2,\ldots,n-1\}}),x_1x_2x_3)}\right)\right\}$$
	By Lemma \ref{reg1} we have
	$$\reg\left(\frac{S}{(I_5(P_{\{2,\ldots,n-1\}}),x_1x_2x_3)}\right)=\begin{cases}
		4p-2,&d=0\\
		4p,&d\in\{1,2,3\}\\
		4p+2,&d\in\{4,5\}
	\end{cases}$$
	and $$\projdim\left(\frac{S}{(I_5(P_{\{2,\ldots,n-1\}}),x_1x_2x_3)}\right)=\begin{cases}
		2p-1&d=0\\
		2p,&d\in\{1,2,3\}\\
		2p+1,&d\in\{4,5\}
	\end{cases}$$
	One may note that
	$$(I_5(P_{\{2,\ldots,n-1\}}),x_1x_2x_3):(x_{n-2}x_{n-1}x_n)=\begin{cases}
		(x_1x_2x_3,x_2x_3x_4),&n=7\\
		(x_1x_2x_3,x_3x_4x_5),&n=8\\
		(x_1x_2x_3,x_4x_5x_6),&n=9\\
		NI(P^2_{n-3}),&n\geq 10
	\end{cases}$$
	Hence
	$$\reg\left(\frac{S}{(I_5(P_{\{2,\ldots,n-1\}}),x_1x_2x_3):(x_{n-2}x_{n-1}x_n)}\right)=\begin{cases}
		2,&n=7\\
		3,&n=8\\
		4,&n=9\\
		\reg\left(\frac{S}{NI(P^2_{n-3})}\right),&n\geq 10
	\end{cases}$$
	and $$\projdim\left(\frac{S}{(I_5(P_{\{2,\ldots,n-1\}}),x_1x_2x_3):(x_{n-2}x_{n-1}x_n)}\right)=
	\begin{cases}
		2,&n\in\{7,8,9\}\\
		\projdim\left(\frac{S}{NI(P^2_{n-3})}\right),&n\geq 10
	\end{cases}$$
	By induction we obtain
	$$\reg\left(\frac{S}{NI(P^2_{n-3})}\right)=\reg\left(\frac{k[x_1,\ldots,x_{n-3}]}{NI(P^2_{n-3})}\right)=\begin{cases}
		4p'',&d''\in\{0,1\}\\
		4p''+1,&d''=2\\
		4p''+2,&d''\in\{3,4\}\\
		4p''+3,&d''=5
	\end{cases}$$
	and $$\projdim\left(\frac{S}{NI(P^2_{n-3})}\right)=\projdim\left(\frac{k[x_1,\ldots,x_{n-3}]}{NI(P^2_{n-3})}\right)=\begin{cases}
		2p''+1,&d''\in\{1,2,3\}\\
		2p''+2,&d''\in\{4,5\}\\
		2p'',&d''=0
	\end{cases}$$
	where $n-3=6p''+d''$.
	Since $n=6p+d$, we have
	$$\reg\left(\frac{S}{NI(P^2_{n-3})}\right)=\begin{cases}
		4p-2,&d\in\{0,1\}\\
		4p-1,&d=2\\
		4p,&d\in\{3,4\}\\
		4p+1,&d=5
	\end{cases}$$
	and $$\projdim\left(\frac{S}{NI(P^2_{n-3})}\right)=\begin{cases}
		2p,&d\in\{1,2,3\}\\
		2p-1,&d=0\\
		2p+1,&d\in\{4,5\}
	\end{cases}$$
	Returning to $$\reg\left(\frac{S}{NI(P^2_{n})}\right)=$$
	$$=\max\left\{\reg\left(\frac{S}{(I_5(P_{\{2,\ldots,n-1\}}),x_1x_2x_3):(x_{n-2}x_{n-1}x_n)}\right)+2,\reg\left(\frac{S}{(I_5(P_{\{2,\ldots,n-1\}}),x_1x_2x_3)}\right)\right\}$$
	we obtain $$\reg\left(\frac{S}{NI(P^2_{n})}\right)=\begin{cases}
		4p,&d\in\{0,1\}\\
		4p+1,&d=2\\
		4p+2,&d\in\{3,4\}\\
		4p+3,&d=5
	\end{cases}$$
	Moreover, by $$\projdim\left(\frac{S}{NI(P^2_{n})}\right)=$$
	$$=\max\left\{\projdim\left(\frac{S}{(I_5(P_{\{2,\ldots,n-1\}}),x_1x_2x_3):(x_{n-2}x_{n-1}x_n)}\right)+1,\projdim\left(\frac{S}{(I_5(P_{\{2,\ldots,n-1\}}),x_1x_2x_3)}\right)\right\}$$
	we obtain $$\projdim\left(\frac{S}{NI(P^2_{n})}\right)=\begin{cases}
		2p,&d=0\\
		2p+1,&d\in\{1,2,3\}\\
		2p+2,&d\in\{4,5\}
	\end{cases}$$
	
	For the cases $n\in\{7,8,9\}$ one may use Singular to check the formulae.   
\end{proof}

An obvious consequence is the following:
\begin{Corollary}\label{depth-reg} Let $NI(P^2_{n})$ be the closed neighborhood ideal of the square of the path graph $P_n$. Then
	$$\reg\left(\frac{S}{NI(P^2_{n})}\right)=\depth\left(\frac{S}{NI(P^2_{n})}\right)=\begin{cases}
		4p,&d\in\{0,1\}\\
		4p+1,&d=2\\
		4p+2,&d\in\{3,4\}\\
		4p+3,&d=5
	\end{cases}$$
\end{Corollary}

We are now able to characterize the closed neighborhood ideals which are Cohen--Macaulay:
\begin{Proposition}\label{CM} Let $NI(P^2_{n})$ be the closed neighborhood ideal of the square of the path graph $P_n$. Then
	$NI(P^2_{n})$ is Cohen--Macaulay if and only if $n\in\{1,6\}.$
\end{Proposition}

\begin{proof}
	Let $n=5q+r=6p+d$, $0\leq r\leq 4$ and $0\leq d\leq 5$.
	\begin{itemize}
		\item For $r=0$ we have $\height(NI(P^2_{n}))=q$. 
		\begin{itemize}
			\item for $d=0$ we have $5q=6p$. In this case $NI(P^2_{n})$ is Cohen--Macaulay if and only if $q=2p$. Therefore $q=p=0$ and $n=0$.
			\item for $d\in\{1,2,3\}$ we have $5q=6p+d$. In this case $NI(P^2_{n})$ is Cohen--Macaulay if and only if $q=2p+1$.
			Therefore $4p=d-5$ and we obtain $p<0$.
			\item for $d\in\{4,5\}$ we have $5q=6p+d$. In this case $NI(P^2_{n})$ is Cohen--Macaulay if and only if $q=2p+2$.
			Therefore $4p=d-10$ and we obtain $p<0$.
		\end{itemize}
		\item For $r\neq 0$ we have $\height(NI(P^2_{n}))=q+1$.
		\begin{itemize}
			\item for $d=0$ we have $5q+r=6p$. In this case $NI(P^2_{n})$ is Cohen--Macaulay if and only if $q+1=2p$. We obtain that $2q=3-r$ and $q\in\mathbb N$ only for $r=1$. Hence it results $q=1$ and $n=6$.
			\item for $d\in\{1,2,3\}$ we have $5q+r=6p+d$. In this case $NI(P^2_{n})$ is Cohen--Macaulay if and only if $q+1=2p+1$.
			Therefore $q=2p$ and from $5q+r=6p+d$ we obtain $2q=d-r$. 
			\begin{itemize}
				\item For $r=1,\ d=1$ we obtain $q=0$ and $n=1$, for $r=1,\ d=2$ we have $q=\frac{1}{2}\notin \mathbb N$ and $r=1,\ d=2$ implies $q=1$ and is not even. 
				\item For $r=2,\ d=2$ we obtain $q=0$ and $n=1$, for $r=2,\ d=3$ we have $q=\frac{1}{2}\notin \mathbb N$ and $r=2,\ d=1$ implies $q<0$.
				\item For $r=3,\ d=3$ we obtain $q=0$ and $n=1$, for $r=1,\ d\in \{1,2\}$ we have $q<0$.
				\item For $r=4$ we get $q<0$.  
			\end{itemize} 
			\item for $d\in\{4,5\}$ we have $5q=6p+d$. In this case $NI(P^2_{n})$ is Cohen--Macaulay if and only if $q+1=2p+2$.
			Therefore $q=2p+1$ and from $5q=6p+d$ we obtain $4p=d-r-5$. Since $1\leq r\leq 4$ we get $d-r-5<0$, contradiction.
		\end{itemize}
	\end{itemize}
\end{proof}

\section{The closed neighborhood ideal of $P_n^2$ and simplicial complexes}

We pay attention now to simplicial complexes in order to determine some other properties of the closed neighborhood ideals. As before, we start by recalling the notions and results that will be used through this section (for more details, please check \cite{BH} or \cite{V}). 

	\textit{A simplicial complex} $\Delta$ on the vertex set $\{x_1,\ldots, x_n\}$ is a collection of subsets (called \textit{faces}) such that any vertex is in $\Delta$ and, if $F$ is a face of $ \Delta$ and $G\subset F$, then $G$ is also a face of $\Delta$.\textit{ The set of facets of} $\Delta$ is denoted by $\mathcal{F}(\Delta)$. \textit{The dimension of a face} $F$ si equal to $|F|-1$. The dimension of $\Delta$ is the maximum of the dimensions of all the faces. Let $f_i$ be the number of faces of dimension $i$ in $\Delta$. The $f$-vector of $\Delta$ is $f(\Delta)=(f_{-1},f_0,\ldots,f_d)$, where $d=\dim(\Delta)$ and $f_{-1}=1$. The\textit{ Euler characteristic of} $\Delta$ is denoted by $\chi(\Delta)$ and is defined as $\chi(\Delta)=\sum\limits_{i=0}^d(-1)^if_i$, while \textit{the reduced Euler characteristic} is $\overline{\chi}(\Delta)=-1+\chi(\Delta)$. A simplicial complex is\textit{ pure} if all its facets have the same dimension. A simplicial complex is \textit{(non-pure) shellable} if its facets can be ordered $F_1,\ldots,F_r$ such that for all $1\leq i<j\leq r$ there are a vertex $v\in F_j\setminus F_i$ and an integer $l<j$ such that $F_j\setminus F_l=\{v\}$ (\cite{BW}).   

Given a simplicial complex $\Delta$ one may associate several simplicial complexes on the same vertex set:
\begin{center}
	\begin{itemize}
		\item \textit{the Alexander dual} of $\Delta$: $\Delta^{\vee}=\{F: F^c\notin \Delta\}$
		
		\item \textit{the complement of} $\Delta$: $\mathcal{F}(\Delta^c)=\{F^c: F\in\mathcal{F}(\Delta)\}$.
\end{itemize}\end{center}

For a simplicial complex on the vertex set $\{x_1,\ldots,x_n\}$, let $S=k[x_1,\ldots,x_n]$ be the polynomial ring over the field $k$. We can consider the following squarefree monomial ideals:
\begin{center}
	\begin{itemize}
		\item \textit{The Stanley--Reisner ideal} of $\Delta$: $I_{\Delta}=\{\mathbf{x}_{F^c}: F\notin \Delta\}$
		
		\item \textit{The facet ideal} of $\Delta$: $I(\Delta^c)=(\mathbf{x}_F: F\in\mathcal{F}(\Delta)\}$.
		
		\item \textit{The complementary ideal }of $\Delta$: $I^c(\Delta)=\left(\frac{x_1\cdots x_n}{\mathbf{x}_F}: F\in\mathcal{F}(\Delta)\right)$.
\end{itemize}\end{center}

Conversely, let $I\subseteq S=k[x_1,\ldots,x_n]$ be a squarefree monomial ideal. One may associate to $I$ several simplicial complexes on the set of vertices $\{x_1,\ldots,x_n\}$:

\begin{center}
	\begin{itemize}
		\item \textit{The Stanley--Reisner simplicial complex} $SR(I)$ as being the simplicial complex such that $I=I_{SR(I)}$.
		\item If $\mathcal{G}(I)=\{\mathbf{x}_{F_1},\ldots, \mathbf{x}_{F_r}\}$, where $F_i\subseteq\{x_1,\ldots,x_n\}$ for all $i$, then one may consider the simplicial complex $\Delta$ whose facets are $F_1,\ldots,F_r$, that is $I=I(\Delta)$.
	\end{itemize}
\end{center}
A simplicial complex $\Delta$ is called (sequentially) Cohen--Macaulay if it's Stanley--Reisner ring $k[\Delta]=S/I_{\Delta}$ is (sequentially) Cohen--Macaulay. A monomial ideal \textit{has linear quotients} if its minimal monomial generators $m_1,\ldots, m_r$ can be ordered such that for all $1\leq i\leq r-1$, the ideal $(m_1\ldots,m_i):m_{i+1}$ is generated by linear forms.

\begin{Theorem}\label{shelllinquot}The following hold:
	\begin{itemize}
		\item[(i)] If $\Delta$ is a  non-pure shellable simplicial complex, then $k[\Delta]$ is sequentially Cohen--Macaulay. $($\cite{St}, \cite[Theorem 6.3.27]{V}$)$
		\item[(ii)] A simplicial complex $\Delta$ is shellable if and only if $I_{\Delta^{\vee}}$ has linear quotients. $($\cite[Theorem 1.4]{HHZ}$)$
	\end{itemize}
	
\end{Theorem}

Throughout this section, we will assume that $n\geq7$.

Let $\mathcal{NI}[P_n^2]$ be the simplicial complex with the vertex set $\{x_1,\ldots,x_n\}$ such that $I(\mathcal{NI}[P_n^2])=NI(P_n^2)$, $n\geq7$. More precisely, the facets of the simplicial complex $\mathcal{NI}[P_n^2]$ are $ F_1=\{x_1,x_2,x_3\},F_{n-4}=\{x_{n-2},x_{n-1},x_n\}, F_i=\{x_{i},x_{i+1},x_{i+2},x_{i+3},x_{i+4}\}$ for all $ 2\leq i\leq n-5.$ 

\begin{Proposition}\label{shellable}{\rm{(The shellability)}}
	For $n\geq 7$, the simplicial complex $\mathcal{NI}[P_n^2]$ is shellable.
\end{Proposition}
\begin{proof}
	We will prove that $\mathcal{NI}[P_n^2]$ is shellable with respect to the order $F_2,\ldots, F_{n-5}$, $F_1,F_{n-4}$.
	
	Let $2\leq i<j\leq n-5$ be two integers. Then there are $x_{j+4}\in F_{j}\setminus F_i$ and $l=j-1<j$ such that $F_j\setminus F_{j-1}=\{x_{j+4}\}$.
	
	For $j=1$ and for all $i\in\{2,\ldots,n-5\}$ there are $x_1\in F_1\setminus F_i$ and $l=2$ such that $F_1\setminus F_2=\{x_1\}$.

	Finally, for $j=n-4$ and for all $i\in\{1,2,\ldots,n-5\}$ there are $x_{n}\in F_{n-4}\setminus F_i$ and $l=n-5$ such that $F_{n-4}\setminus F_{n-5}=\{x_n\}$.
\end{proof}
We recall that a vertex of a simplicial complex $\Delta$ is called \textit{a free vertex} if it belongs to exactly one facet of $\Delta$.
According to \cite[Definition 4.1]{Z1} a simplicial complex $\Delta$\textit{ has the free vertex property} if it satisfies the following conditions:
\begin{enumerate}
	\item $\Delta$ is a simplex, or
	\item $\Delta$ has a free vertex $x$ such that both $\Delta\setminus\langle F\rangle$ and $\Delta\setminus x$ also have free vertices, where $F$ is the only facet of $\Delta$ containing $x$, $\Delta\setminus\langle F\rangle$ is the simplicial complex obtained by removing the facet $F$ from $\Delta$, and $\Delta \setminus x= \{G\in \Delta: x\notin G\}.$
\end{enumerate}  

\begin{Proposition}\label{freevertex}{\rm{(The free vertex property)}}
	The complex $\mathcal{NI}[P_n^2]$ has the free vertex property.
\end{Proposition}

\begin{proof}
	Indeed, one may easily note that $x_1\in F_1$ is a free vertex. The statement follows since both $\mathcal{NI}[P_n^2]\setminus\langle F_1\rangle=\langle F_2,\ldots, F_{n-4}\rangle$ and $\mathcal{NI}[P_n^2]\setminus x_1=\langle F_2,\ldots, F_{n-4}\rangle$ have $x_n$ as free vertex.
\end{proof}
\begin{Theorem}\cite[Theorem 4.8]{Z1}\label{shell}
	Let $\Delta$ be a simplicial complex, $I(\Delta)$ its facet ideal and $SR(\Delta)$ the simplicial complex with the property that $I(\Delta)=I_{SR(\Delta)}$. If $\Delta$ has the free vertex property, then $SR(\Delta)$ is shellable.
\end{Theorem}

 The Stanley--Reisner complex of $NI(G)$ is\textit{ the dominance complex of a graph $G$} which was studied by Matsushita and Wakatsuki \cite{MW}. We obtain that, for $P_n^2$, the dominance complex is sequentially Cohen--Macaulay Theorem \ref{SCM1}.
\begin{Theorem}\cite[Theorem 4.9]{Z1}\label{SCM}
	If the simplicial complex $\Delta$ on a vertex set $V$ has the free vertex property, then
	$S/I(\Delta)$ is sequentially Cohen--Macaulay.
\end{Theorem}
\begin{Theorem}\label{SCM1}{\rm{(Sequentially Cohen--Macaulay)}} The following statements hold:
	\begin{itemize}
		\item[(i)]  $S/NI(P_{n}^2)$ is sequentially Cohen--Macaulay.
		\item[(ii)] $S/I_{SR(\mathcal{NI}[P_n^2])}$ is sequentially Cohen--Macaulay.
	\end{itemize}

\end{Theorem}
\begin{proof}
	(i) The statement follows easy by Theorem \ref{SCM}, Proposition \ref{freevertex} and by using the fact that $I(\mathcal{NI}[P_n^2])=NI(P_{n}^2)$.
	
	(ii) This follows directly from Theorem \ref{shelllinquot}(i)
\end{proof}
The following result will help us to compute the big height of the closed neghborhood ideal. We recall that the\textit{ big height} of a monomial ideal $I$ is denoted by $\bight(I)$ and is given by the maximal height of the minimal associated primes of $I$. 

\begin{Lemma}\label{SCMbh}\cite[Corollary 6.4.20]{V},\cite[Corollary 3.33]{MV} Let $I$ be a squarefree monomial ideal of $S$ and $S/I$ is sequentially Cohen--Macaulay. Then $\projdim (S/I)=\bight(I)$.
	\end{Lemma}
	
We can now determine the big height of the closed neighborhood complex of $P_n^2$, and the maximal size of a minimal domination set of $P_n^2$
\begin{Proposition}{\rm{(The big height)}}
Let $n\geq7$ be an integer, and $n=6p+d$ with $0\leq d\leq 5$. Then $$\bight(NI(P_n^2))=\gamma'_{P_n^2}=\begin{cases}
		2p,&d=0\\
		2p+1,&d\in\{1,2,3\}\\
		2p+2,&d\in\{4,5\}
	\end{cases}$$

\end{Proposition}
\begin{proof} The first equality holds by Proposition \ref{domnum}.
	The formula follows since $S/NI(P_n^2)$ is sequentially Cohen--Macaulay (Theorem \ref{SCM1}(i)) and we use Lemma \ref{SCMbh} and the formulae of projective dimension given in Theorem \ref{pdreg}.
\end{proof}

We can also prove that the complementary ideal of $P_n^2$ has linear quotients.
\begin{Corollary} The ideal $I^c(P_{n}^2)=(x_1\cdots x_{n-3},x_4\cdots x_n)+\left(\frac{x_1\cdots x_n}{x_i\cdots x_{i+4}}:2\leq i\leq n-5\right)$ has linear quotients.
\end{Corollary}
\begin{proof} By using \cite[Lemma 1.2]{HHZ}, we get $I_{\mathcal{NI}[P_n^2]^\vee}=I(\mathcal{NI}[P_n^2]^c)=I^c(P_{n}^2)$ and since $\mathcal{NI}[P_n^2]$ is shellable, the statement follows by Theorem \ref{shelllinquot}(ii).
	\end{proof}

\begin{Proposition}\label{fvector} {\rm{(The $f$-vector)}}
	The $f$-vector of $\mathcal{NI}[P_n^2]$ is $$f(\mathcal{NI}[P_n^2])=(1,n,4n-14,6n-30,4n-23,n-6).$$
\end{Proposition}
\begin{proof}
	The statement follows by direct computations.
We have	$f_{-1}=1$ and $f_0=|V(\mathcal{NI}[P_n^2])|=n$. Since $f_1$ is the number of edges, we have 10 edges in $F_2$, 2 edges in $F_1$, we removed $\{x_2,x_3\}$. Each of the facets $F_3,\ldots,F_{n-5}$ contribute with $4$ edges each of them (since each edge contains only one new vertex which is connected to each one of the previous $4$ vertices). Finally $F_{n-4}$ contributes with only 2 edges. Hence $$f_1=2+10+4(n-7)+2=4n-14.$$ 

For $f_2$, by using the same technique and taking into account that $F_1$ and $F_{n-4}$ are of dimension 2, we get: $$f_2=2+{5\choose3}+{4\choose 2}\cdot(n-7)=2+10+6(n-7)=6n-30.$$

In the same way, we get $$f_3={5\choose4}+(n-7)\cdot{4\choose3}=4n-23$$and$$f_4=n-6.$$
\end{proof}
We can now easily compute the $h$-vector:
\begin{Corollary}\label{hvect} {\rm{(The $h$-vector)}}
		The $h$-vector of $\mathcal{NI}[P_n^2]$ is $(1,n-5,-4,2,0,0)$.
\end{Corollary}
\begin{proof}
	For the proof one may just apply the formulae which relate the $f$-vector and the $h$-vector, namely
	$$h_i(\Delta)=\sum\limits_{j=0}^i(-1)^{i-j}{{d-j}\choose {i-j}}f_{j-1}(\Delta),$$for all $i=0,\ldots,d$ (here $\dim(\Delta)=d-1$), taking into account that in our case $d=5$ and use Proposition \ref{fvector}.
\end{proof}
The above results allow us to compute the (reduced) Euler characteristic and the multiplicity. 
\begin{Proposition}
	For $n\geq7$ we have $\chi(\mathcal{NI}[P_n^2])=1$, $\overline{\chi}(\mathcal{NI}[P_n^2])=0$ and $e(S/NI(P_n^2))=n-6$.
\end{Proposition}
\begin{proof}
	Since $\chi(\mathcal{NI}[P_n^2])=f_0-f_1+f_2-f_3+f_4$, by using Proposition \ref{fvector} we have $$\chi(\mathcal{NI}[P_n^2])=n-(4n-14)+(6n-30)-(4n-23)+(n-6)=1.$$ Now $\overline{\chi}(\mathcal{N}I[P_n^2])=\chi(\mathcal{NI}[P_n^2])-1=0$. For the multiplicity, the formula follows easy by using Corollary \ref{hvect}:
	$e(S/NI(P_n^2)=\sum\limits_{i=0}^dh_i=1+(n-5)-4+2=n-6$.
\end{proof}

\end{document}